%% file: intcomp.tex
\documentclass[preprint,twosides,times]{elsarticle}
\usepackage[utf8]{inputenc}
\usepackage{amsmath}
\usepackage{amssymb}
\usepackage{amsthm}
\usepackage{stmaryrd}
\usepackage{bussproofs}
\usepackage{array}
\usepackage{hyperref}
\usepackage{float}
\floatstyle{boxed}
\restylefloat{table}

\input{macros.tex}

\begin{document}

\begin{frontmatter}
  \title{Continuation-passing Style Models Complete for Intuitionistic Logic}
  \author{Danko Ilik}
  \address{University ``Goce Delčev'' -- Štip\\
    Address: Faculty of Informatics, PO Box 201, Štip, Macedonia\\
    E-mail: dankoilik@gmail.com}
  
  \begin{abstract}A class of models is presented, in the form of continuation monads polymorphic for first-order individuals, that is sound and complete for minimal intuitionistic predicate logic. The proofs of soundness and completeness are constructive and the computational content of their composition is, in particular, a $\beta$-normalisation-by-evaluation program for simply typed lambda calculus with sum types. Although the inspiration comes from Danvy's type-directed partial evaluator for the same lambda calculus, the there essential use of delimited control operators (i.e. computational effects) is avoided. The role of polymorphism is crucial -- dropping it allows one to obtain a notion of model complete for classical predicate logic. The connection between ours and Kripke models is made through a strengthening of the Double-negation Shift schema. 
  \end{abstract}

  \begin{keyword}
    intuitionistic logic \sep completeness \sep Kripke models \sep Double-negation Shift \sep normalization by evaluation
  
    \MSC 03B20 \sep 03B35 \sep 03B40 \sep 68N18 \sep 03F55 \sep 03F50 \sep 03B55
  \end{keyword}

\end{frontmatter}

\section{Introduction}

Although Kripke models are standard semantics for intuitionistic logic, there is as yet no (simple) constructive proof of their completeness when one considers all logical connectives. While Kripke's original proof \cite{Kripke1963} was classical, Veldman gave an intuitionistic one \cite{Veldman1976} by using Brouwer's Fan Theorem to handle disjunction and the existential quantifier. To see what the computational content behind Veldman's proof is, one might consider a realisability interpretation of the Fan Theorem (for example \cite{BergerO2005}), but, all known realisers being defined by general recursion, due to the absence of an elementary proof of their termination, it is not clear whether one can think of the program using them as a constructive proof or not.

On the other hand, a connection between normalisation-by-evaluation (NBE) \cite{BergerS1991} for simply typed lambda calculus, $\lambda^\to$, and completeness for Kripke models for the fragment $\{\wedge,\Rightarrow,\forall\}$ has been made \cite{CCoquand1993,HerbelinLee2009}. We review this connection in Section~\ref{sec:nbe}. There we also look at Danvy's extension \cite{Danvy1996} of NBE from $\lambda^\to$ to $\lambda^{\to\vee}$, simply typed lambda calculus with sum types. Even though Danvy's algorithm is simple and elegant, he uses the full power of delimited control operators which do not yet have a typing system that permits to understand them logically. We deal with that problem in Section~\ref{sec:completeness}, by modifying the notion of Kripke model so that we can give a proof of completeness for full intuitionistic logic in continuation-passing style, that is, without relying on having delimited control operators in our meta-language. In Section~\ref{sec:correctness}, we extract the algorithm behind the given completeness proof, a $\beta$-NBE algorithm for $\lambda^{\to\vee}$. In Section~\ref{sec:variants}, we stress the importance of our models being dependently typed, by comparing them to similar models that are complete for classical logic \cite{IlikLH2010}. We there also relate our and Kripke models by showing that the two are equivalent in presence of a strengthening of the Double-negation Shift schema \cite{Spector,TroelstraVD1}. We conclude with Section~\ref{sec:conclusion} by mentioning related work.

The proofs of Section~\ref{sec:completeness} have been formalised in the Coq proof assistant in \cite{intcomp_formalisation}, which also represents an implementation of the NBE algorithm.

\section{Normalisation-by-Evaluation as Completeness}\label{sec:nbe}

In \cite{BergerS1991}, Berger and Schwichtenberg presented a proof of normalisation of $\lambda^\to$ which does not involve reasoning about the associated reduction relation. Instead, they interpret $\lambda$-terms in a domain, or ambient meta-language, using an evaluation function,
\[
\llbracket-\rrbracket : \Lambda \to D,
\]
and then they define an inverse to this function, which from the denotation in $D$ directly extracts a term in $\beta\eta$-long normal form.  The inverse function $\downarrow$, called \emph{reification}, is defined by recursion on the type $\tau$ of the term, at the same time defining an auxiliary function $\uparrow$, called \emph{reflection}:
\begin{align*}
  \downarrow^\tau &: D \to \Lambda\text{-nf} \\
  \downarrow^\tau &:= a\mapsto a & \tau\text{-atomic}\\
  \downarrow^{\tau\to\sigma} &:= S\mapsto \lambda a. \downarrow^\sigma(S\cdot\uparrow^\tau a) & a\text{-fresh}\\
  \\
  \uparrow^\tau &: \Lambda\text{-ne} \to D\\
  \uparrow^\tau &:= a\mapsto a & \tau\text{-atomic}\\
  \uparrow^{\tau\to\sigma} &:= e\mapsto S\mapsto \uparrow^\sigma e (\downarrow^\tau S)
\end{align*}
Here, $S$ ranges over members of $D$, and we used $\mapsto$ and $\cdot$ for abstraction and application at the meta-level. The subclasses of normal and neutral $\lambda$-terms are given by the following inductive definition.
\begin{align*}
  \Lambda\text{-nf} \ni r &:= \lambda a^\tau.r^\sigma ~|~ e^\tau & \lambda\text{-terms in normal form}\\
  \Lambda\text{-ne} \ni e &:= a^\tau ~|~ e^{\tau\to\sigma} r^\tau& \text{neutral } \lambda\text{-terms}
\end{align*}\label{lambda_neutral}

It was a subsequent realisation of Catarina Coquand \cite{CCoquand1993}, that the evaluation algorithm $\llbracket\cdot\rrbracket$ is also the one underlying the Soundness Theorem for minimal intuitionistic logic (with $\Rightarrow$ as the sole logical connective) with respect to Kripke models, and that the reification algorithm $\downarrow$ is also the one underlying the corresponding Completeness Theorem. 

\begin{definition}\label{kripke_model} A \emph{Kripke model} is given by a preorder $(K,\le)$ of \emph{possible worlds}, a binary relation of \emph{forcing} $(-) \Vdash (-)$ between worlds and atomic formulae, and a \emph{family of domains of quantification} $D(-)$, such that,
  \begin{align*}
    \text{ for all } w'\ge w&, w\Vdash X \to w'\Vdash X, \text{ and }\\
    \text{ for all } w'\ge w&, D(w)\subseteq D(w').
  \end{align*}
  The relation of forcing is then extended from atomic to composite formulae by the clauses:
  \begin{align*}
    w\Vdash A\wedge B & := w\Vdash A \text{ and } w\Vdash B\\
    w\Vdash A \lor B & := w\Vdash A \text{ or } w \Vdash B \\
    w\Vdash A\Rightarrow B & := \text{ for all } w'\ge w, w'\Vdash A \Rightarrow w'\Vdash B\\
    w\Vdash \forall x.A(x) & := \text{ for all } w'\ge w \text{ and } t\in D(w'), w'\Vdash A(t)\\
    w\Vdash \exists x.A(x) & := \text{ for some } t\in D(w), w\Vdash A(t)\\
    w\Vdash \bot & := \text{ false } \\
    w\Vdash \top & := \text{ true }
  \end{align*}
\end{definition}

More precisely, the following well-known statements hold and their proofs have been machine-checked \cite{CCoquand2002,HerbelinLee2009} for the logic fragment generated by the connectives $\{\Rightarrow,\wedge,\forall\}$.

\begin{theorem}[Soundness] If $\Gamma\vdash p:A$ then, in any Kripke model, for any world $w$, if $w\Vdash\Gamma$ then $w\Vdash A$.
\end{theorem}
\begin{proof} By a simple induction on the length of the derivation.\end{proof}

\begin{theorem}[Model Existence or Universal Completeness]\label{intcalU} There is a model $\calU$ (the ``universal model'') such that, given a world $w$ of $\calU$, if $w\Vdash A$, then there exists a term $p$ and a derivation in normal form $w\vdash p:A$.
\end{theorem}
\begin{proof}
  The universal model $\calU$ is built by setting:
  \begin{itemize}
  \item $K$ to be the set of contexts $\Gamma$;
  \item ``$\le$'' to be the subset relation of contexts;
  \item ``$\Gamma \Vdash X$'' to be the set of derivations in normal form $\Gamma\vdash^\text{nf} X$, for $X$ an atomic formula.
  \end{itemize}

  One then proves simultaneously, by induction on the complexity of $A$, that the two functions defined above, reify ($\downarrow$) and reflect ($\uparrow$), are correct, that is, that $\downarrow$ maps a member of $\Gamma\Vdash A$ to a normal proof term (derivation) $\Gamma\vdash p:A$, and that $\uparrow$ maps a neutral term (derivation) $\Gamma\vdash e:A$ to a member of $\Gamma\Vdash A$.\qed
\end{proof}

\begin{corollary}[Completeness (usual formulation)]\label{intnbe} If in any Kripke model, at any world $w$, $w\Vdash \Gamma$ implies $w\Vdash A$, then there exists a term $p$ and a derivation $\Gamma\vdash p:A$.
\end{corollary}
\begin{proof}
If $w\Vdash\Gamma \to w\Vdash A$ in any Kripke model, then also $w\Vdash\Gamma \to w\Vdash A$ in the model $\calU$ above. Since from the $\uparrow$-part of Theorem~\ref{intcalU} we have that $\Gamma\Vdash\Gamma$, then from the $\downarrow$-part of the same theorem there exists a term $p$ such that $\Gamma\vdash p:A$.\qed
\end{proof}

If one wants to extend this technique for proving completeness for Kripke models to the rest of the intuitionistic connectives, $\bot$, $\vee$ and $\exists$, the following meta-mathematical problems appear, which have been investigated in the middle of the last century. At that time, Kreisel, based on observations of G\"odel, showed (Theorem 1 of \cite{Kreisel1962}) that for a wide range of intuitionistic semantics, into which Kripke's can also be fit:
\begin{itemize}
\item If one can prove the completeness for the negative fragment of formulae (built using $\wedge, \bot, \Rightarrow, \forall$, and negated atomic formulae, $X\Rightarrow\bot$) then one can prove Markov's Principle. In view of Theorem~\ref{intcalU}, this implies that having a completeness proof cover $\bot$ means being able to prove Markov's Principle -- which is known to be independent of many constructive logical systems, like Heyting Arithmetic or Constructive Type Theory.
\item If one can prove the completeness for all connectives, i.e. including $\vee$ and $\exists$, then one can prove a strengthening\footnote{A special case of D-DNS$^+$ from page \pageref{ddnsplus}.} of the Double-negation Shift schema on $\Sigma^0_1$-formulae, which is also independent because it implies Markov's Principle.
\end{itemize}
We mentioned that Veldman \cite{Veldman1976} used Brouwer's Fan Theorem to handle $\vee$ and $\exists$, but to handle $\bot$ he included in his version of Kripke models an ``exploding node'' predicate, $\Vdash_\bot$ and defined $w\Vdash\bot := w\Vdash_\bot$. 
We remark in passing that Veldman's modification does not defy Kripke original definition, but only makes it more regular: if in Definition~\ref{kripke_model} one considers $\bot$ as an atomic formula, rather than a composite one, one falls back to Veldman's definition.

One can also try to straightforwardly extend the NBE-Completeness proof to cover disjunction (the existential quantifier is analogous) and see what happens. If one does that, one sees that a problem appears in the case of reflection of sum, $\uparrow^{A\vee B}$. There, given a neutral $\lambda$-term that derives $A\vee B$, one is supposed to prove that $w\Vdash A\vee B$ holds, which by definition means to prove that either $w\Vdash A$ or $w\Vdash B$ holds. But, since the input $\lambda$-term is neutral, it represents a blocked computation from which we will only be able to see whether $A$ or $B$ was derived, once we substitute values for the contained free variables that block the computation.

That is where the 
solution of Olivier Danvy appears. In \cite{Danvy1996}, he used the full power\footnote{We say ``full power'' because his usage of delimited control operators is strictly more powerful than what is possible with (non-delimited) control operators like call/cc. Danvy's program makes non-tail calls with continuations, while in the CPS translation of a program that uses call/cc all continuation calls are tail calls.} of the delimited control operators shift ($\shift{k}{p}$) and reset ($\#$) \cite{DanvyF1989} to give the following normalisation-by-evaluation algorithm for $\lambda^{\to\vee}$:
\label{intcomp_tdpe_algo}
\begin{align*}
  \downarrow^\tau &: D \to \Lambda\text{-nf} \\
  \downarrow^\tau &:= a\mapsto a & \tau\text{-atomic}\\
  \downarrow^{\tau\to\sigma} &:= S\mapsto \lambda a. \reset \downarrow^\sigma(S\cdot\uparrow^\tau a) & a\text{-fresh}\\
  \downarrow^{\tau\vee\sigma} &:= S\mapsto
  \left\{
    \begin{array}{ll}
      \iota_1 (\downarrow^\tau S') &, \text{ if } S=\inl\cdot S'\\
      \iota_2 (\downarrow^\sigma S') &, \text{ if } S=\inr\cdot S'\\
    \end{array}
  \right.\\
  \uparrow^\tau &: \Lambda\text{-ne} \to D\\
  \uparrow^\tau &:= a\mapsto a & \tau\text{-atomic}\\
  \uparrow^{\tau\to\sigma} &:= e\mapsto S\mapsto \uparrow^\sigma e (\downarrow^\tau S)\\
  \uparrow^{\tau\vee\sigma} &:= e\mapsto \shift{\kappa}{\caseof{e}{a_1}{\reset \kappa\cdot(\inl\cdot(\uparrow^\tau a_1))}{a_2}{\reset \kappa\cdot(\inr\cdot(\uparrow^\sigma a_2))}} & a_i\text{-fresh}
\end{align*}
We characterise explicitly normal and neutral $\lambda$-terms by the following inductive definitions.
\begin{align*}
  \Lambda\text{-nf} \ni r &:= e^\tau ~|~ \lambda a^\tau.r^\sigma ~|~ \iota^\tau_1 r ~|~ \iota^\tau_2 r\\
  \Lambda\text{-ne} \ni e &:= a^\tau ~|~ e^{\tau\to\sigma} r^\tau ~|~ \caseof{e^{\tau\vee\sigma}}{a_1^\tau}{r_1^\rho}{a_2^\sigma}{r_2^\rho}
\end{align*}\label{intcomp_lambda_neutral}

Given Danvy's NBE algorithm, which is simple and appears correct\footnote{For more details on the computational behaviour of shift/reset and the algorithm itself, we refer the reader to the original paper \cite{Danvy1996} and to Section 3.2 of \cite{IlikThesis}.}, does this mean that we can obtain a constructive proof of completeness for Kripke models if we permit delimited control operators in our ambient meta-language? Unfortunately, not, or not yet, because the available typing systems for them are either too complex (type-and-effect systems \cite{DanvyF1989} change the meaning of implication), or do not permit to type-check the algorithm as a completeness proof (for example the typing system from \cite{FilinskiThesis}, or the one from Chapter 4 of \cite{IlikThesis}).

\section{Kripke-CPS Models and Their Completeness}\label{sec:completeness}

However, there is a close connection between shift and reset, and the continuation-passing style (CPS) translations \cite{DanvyF1990}. We can thus hope to give a normalisation-by-evaluation proof for full intuitionistic logic in continuation-passing style. 

In this section we present a notion of model that we developed following this idea, by suitably inserting continuations into the notion of Kripke model. We prove that the new models are sound and complete for full intuitionistic predicate logic. 

\begin{definition}\label{intcomp_model_def} An \emph{Intuitionistic Kripke-CPS model (IK-CPS)} is given by:
  \begin{itemize}
  \item a preorder $(K, \le)$ of \emph{possible worlds};
  \item a \fbox{binary} relation on worlds $(-) \explod{(-)}$ labelling a world as \emph{exploding};
  \item a binary relation $(-) \svd (-)$ of \emph{strong forcing} between worlds and atomic formulae, such that
    \[
    \text{ for all } w'\ge w, w\svd X \to w'\svd X,
    \]
  \item and a domain of quantification $D(w)$ for each world $w$, such that
    \[
    \text{ for all } w'\ge w, D(w)\subseteq D(w').
    \]
  \end{itemize}
  The relation $(-) \svd (-)$ of \emph{strong forcing} is \emph{extended from atomic to composite formulae} inductively and by simultaneously defining \fbox{one} new relation, (non-strong) forcing:
  \begin{itemize}
  \item[$\star$] A formula $A$ is \emph{forced} in the world $w$ (notation $w\Vdash A$) if, \fbox{for any formula $C$,}
    \[
    \forall w'\ge w.~ \left(\forall w''\ge w'.~ w''\sforces A \to w''\explod{C}\right) \to w'\explod{C};
    \]
  \item $w\svd A\wedge B$ if $w\Vdash A$ and $w\Vdash B$;
  \item $w\svd A\vee B$ if $w\Vdash A$ or $w\Vdash B$;
  \item $w\svd A\Rightarrow B$ if for all $w'\ge w$, $w\Vdash A$ implies $w\Vdash B$;
  \item $w\svd \forall x. A(x)$ if for all $w'\ge w$ and all $t\in D(w')$, $w'\Vdash A(t)$;
  \item $w\svd \exists x. A(x)$ if $w\Vdash A(t)$ for some $t\in D(w)$.
  \end{itemize}
\end{definition}

\begin{remark} Certain details of the definition have been put into boxes to facilitate the comparison carried out in Section \ref{sec:variants}.
\end{remark}

\begin{lemma}\label{intcomp_model_monotone} Strong forcing and (non-strong) forcing are monotone in any IK-CPS model, that is, given $w'\ge w$, $w\svd A$ implies $w'\svd A$, and $w\Vdash A$ implies $w'\Vdash A$.
\end{lemma}
\begin{proof} Monotonicity of strong forcing is proved by induction on the complexity of the formula, while that of forcing is by definition. The proof is easy and available in the Coq formalisation.
\end{proof}

\begin{lemma}\label{intcomp_model_monad} The following monadic operations are definable for IK-CPS models:
  \begin{description}
  \item[``unit'' $\eta(\cdot)$ ] $w\sforces A \to w\Vdash A$
  \item[``bind'' $(\cdot)^*(\cdot)$ ] $(\forall w'\ge w.~ w'\sforces A \to w'\Vdash B) \to w\Vdash A \to w\Vdash B$
  \end{description}
\end{lemma}
\begin{proof} Easy, using Lemma \ref{intcomp_model_monotone}.
  If we leave implicit the handling of formulae $C$, worlds, and monotonicity, we have the following procedures behind the proofs.
  \begin{align*}
    \eta (\alpha) &= \kappa\mapsto \kappa \cdot \alpha\\
    (\phi)^*(\alpha) &= \kappa\mapsto \alpha\cdot(\beta\mapsto \phi\cdot\beta\cdot\kappa)
  \end{align*}
  \qed
\end{proof}

\begin{table}[h!]
  \centering
  \begin{tabular}{ m{5cm} m{6cm} }
    \begin{prooftree}
      \axc{$(a:A) \in \Gamma$}
      \uic{$\Gamma\vdash a:A$}{\textsc{Ax}}
    \end{prooftree}
    &

    \\
    \begin{prooftree}
      \axc{$\Gamma\vdash p:A_1$}
      \axc{$\Gamma\vdash q:A_2$}
      \bic{$\Gamma\vdash (p,q):A_1\wedge A_2$}{$\wedge_I$}
    \end{prooftree}
    &
    \begin{prooftree}
      \axc{$\Gamma\vdash p:A_1\wedge A_2$}
      \uic{$\Gamma\vdash \pi_i p:A_i$}{$\wedge^i_E$}
    \end{prooftree}
    \\
    \begin{prooftree}
      \axc{$\Gamma\vdash p:A_i$}
      \uic{$\Gamma\vdash \iota_i p:A_1\vee A_2$}{$\vee^i_I$}
    \end{prooftree}
    & \\
    \multicolumn{2}{ m{11cm} }{
      \begin{prooftree}
        \axc{$\Gamma\vdash p:A_1\vee A_2$}
        \axc{$\Gamma, a_1:A_1\vdash q_1:C$}
        \axc{$\Gamma, a_2:A_2\vdash q_2:C$}
        \tic{$\Gamma\vdash \caseof{p}{a_1}{q_1}{a_2}{q_2}:C$}{$\vee_E$}
      \end{prooftree}
    }
    \\
    \begin{prooftree}
      \axc{$\Gamma, a:A_1\vdash p:A_2$}
      \uic{$\Gamma\vdash \lambda a.p:A_1\Rightarrow A_2$}{$\Rightarrow_I$}
    \end{prooftree}
    &
    \begin{prooftree}
      \axc{$\Gamma\vdash p:A_1\Rightarrow A_2$}
      \axc{$\Gamma\vdash q:A_1$}
      \bic{$\Gamma\vdash p q:A_2$}{$\Rightarrow_E$}
    \end{prooftree}
    \\
    \begin{prooftree}
      \axc{$\Gamma\vdash p:A(x)$}
      \axc{$x\text{-fresh}$}
      \bic{$\Gamma\vdash \lambda x.p:\forall x. A(x)$}{$\forall_I$}
    \end{prooftree}
    &
    \begin{prooftree}
      \axc{$\Gamma\vdash p:\forall x.A(x)$}
      \uic{$\Gamma\vdash p t:A(t)$}{$\forall_E$}
    \end{prooftree}
    \\
    \begin{prooftree}
      \axc{$\Gamma\vdash p:A(t)$}
      \uic{$\Gamma\vdash (t,p):\exists x.A(x)$}{$\exists_I$}
    \end{prooftree}
    & \\
    \multicolumn{2}{ m{11cm} }{
      \begin{prooftree}
        \axc{$\Gamma\vdash p:\exists x.A(x)$}
        \axc{$\Gamma, a:A(x)\vdash q:C$}
        \axc{$x\text{-fresh}$}
        \tic{$\Gamma\vdash \destas{p}{x}{a}{q}:C$}{$\exists_E$}
      \end{prooftree}
    }
    \\
  \end{tabular}
  
  \caption[MQC with proof terms]{Proof term annotation for the natural deduction system of minimal intuitionistic predicate logic (MQC)}
  \label{tab:iqc_terms}
\end{table}

With Table~\ref{tab:iqc_terms}, we fix a derivation system and proof term notation for minimal intuitionistic predicate logic.   There are two kinds of variables, proof term variables $a,b,\ldots$ and individual (quantifier) variables $x,y,\ldots$. Individual constants are denoted by $t$. We rely on these conventions to resolve the apparent ambiguity of the syntax: the abstraction $\lambda a.p$ is a proof term for implication, while $\lambda x.p$ is a proof term for $\forall$; $(p,q)$ is a proof term for $\wedge$, while $(t,q)$ is a proof term for $\exists$.

We supplement the characterisation of normal and neutral terms from page~\pageref{intcomp_lambda_neutral}:
\begin{align*}
  \Lambda\text{-nf} \ni r := & e ~|~ \lambda a.r ~|~ \iota_1 r ~|~ \iota_2 r ~|~ (r_1,r_2) ~|~ \lambda x.r ~|~ (t,r) \\
  \Lambda\text{-ne} \ni e := & a ~|~ e r ~|~ \caseof{e}{a_1}{r_1}{a_2}{r_2} ~|~ \pi_1 e ~|~ \pi_2 e ~|~ e t ~|~ \\
   & ~ \destas{e}{x}{a}{r}
\end{align*}

As before, let $w\Vdash \Gamma$ denote that all formulae from $\Gamma$ are forced.

\begin{theorem}[Soundness]\label{int_soundness}
  If $\Gamma\vdash p:A$, then, in any world $w$ of any IK-CPS model, if $w\Vdash\Gamma$, then $w\Vdash A$.
\end{theorem}
\begin{proof}
This is proved by a simple induction on the length of the derivation. We give the algorithm behind it in section ~\ref{sec:correctness}.
\end{proof}

\begin{remark}\label{polymorphism} The condition ``for all formula $C$'' in Definition \ref{intcomp_model_def} is only necessary for the soundness proof to go through, more precisely, the cases of elimination rules for $\vee$ and $\Rightarrow$. The completeness proof goes through even if we define forcing by
\[
\forall w'\ge w.~ \left(\forall w''\ge w'.~ w''\sforces A \to w''\explod{A}\right) \to w'\explod{A}.
\]
\end{remark}

\begin{definition}The \emph{Universal IK-CPS model} $\calU$\ is obtained by setting:
  \begin{itemize}
  \item $K$ to be the set of contexts $\Gamma$ of MQC;
  \item $\Gamma\le\Gamma'$ iff $\Gamma\subseteq\Gamma'$;
  \item $\Gamma\sforces X$ iff there is a derivation in normal form of $\Gamma\vdash X$ in MQC, where $X$ is an atomic formula;
  \item $\Gamma\explod{C}$ iff there is a derivation in normal form of $\Gamma\vdash C$ in MQC;
  \item for any $w$, $D(w)$ is a set of individuals for MQC (that is, $D(-)$ is a constant function from worlds to sets of individuals).
  \end{itemize}
  $(-)\sforces (-)$ is monotone because of the weakening property for intuitionistic ``$\vdash$''.
\end{definition}

\begin{remark} The difference between strong forcing ``$\sforces$'' and the exploding node predicate ``$\explod{C}$'' in $\calU$ is that the former is defined on atomic formulae, while the latter is defined on any kind of formulae.
\end{remark}

\begin{lemma}\label{intcomp_model_run} We can also define the monadic \emph{``run''} operation on the universal model $\calU$, \emph{for atomic formulae} $X$:
\[
\mu(\cdot): w\Vdash X \to w\sforces X.
\]
\end{lemma}
\begin{proof}
By setting $C:=A$ and applying the identity function.
\end{proof}

\begin{theorem}[Completeness for $\calU$]\label{int_completeness}
  For any closed formula $A$ and closed context $\Gamma$, the following hold for $\calU$:
  \begin{align}
    \tag{$\downarrow$}\label{int_reify}\Gamma\Vdash A & \longrightarrow \{p ~|~ \Gamma\vdash p:A\} & (\text{``reify''})\\
    \tag{$\uparrow$}\label{int_reflect}\Gamma\vdash e:A & \longrightarrow \Gamma\Vdash A & (\text{``reflect''})
  \end{align}
  Moreover, the target of ($\downarrow$) is a normal term, while the source of ($\uparrow$) is a neutral term.
\end{theorem}
\begin{proof} We prove simultaneously the two statements by induction on the complexity of formula $A$.

We skip writing the proof term annotations, and write just $\Gamma\vdash A$ instead of ``there exists $p$ such that $\Gamma\vdash p:A$'', in order to decrease the level of detail. The algorithm behind this proof that concentrates on proof terms is given in Section \ref{sec:correctness}.

\emph{Base case.} ($\downarrow$) is by ``run'' (Lemma \ref{intcomp_model_run}), ($\uparrow$) is by ``unit'' (Lemma \ref{intcomp_model_monad}).

\emph{Induction case for $\wedge$.} Let $\Gamma\Vdash A\wedge B$ i.e.
\[
\idcont{C}{\Gamma}{\Gamma''\Vdash A \text{ and } \Gamma''\Vdash B}.
\]
We apply this hypothesis by setting $C:=A\wedge B$ and $\Gamma':=\Gamma$, and then, given $\Gamma''\ge\Gamma$ s.t. $\Gamma''\Vdash A \text{ and } \Gamma''\Vdash B$, we have to derive $\Gamma''\vdash A\wedge B$. But, this is immediate by applying the $\wedge_I$ rule and the induction hypothesis ($\downarrow$) twice, for $A$ and for $B$.

Let $\Gamma\vdash A\wedge B$ be a neutral derivation. We prove $\Gamma\Vdash A\wedge B$ by applying unit (Lemma \ref{intcomp_model_monad}), and then applying the induction hypothesis ($\downarrow$) on $\wedge^1_I$, $\wedge^2_I$, and the hypothesis.

\emph{Induction case for $\vee$.} Let $\Gamma\Vdash A\vee B$ i.e.
\[
\idcont{C}{\Gamma}{\Gamma''\Vdash A \text{ or } \Gamma''\Vdash B}.
\]
We apply this hypothesis by setting $C:=A\vee B$ and $\Gamma':=\Gamma$, and then, given $\Gamma''\ge\Gamma$ s.t. $\Gamma''\Vdash A \text{ or } \Gamma''\Vdash B$, we have to derive $\Gamma''\vdash A\vee B$. But, this is immediate, after a case distinction, by applying the $\vee^i_I$ rule and the induction hypothesis ($\downarrow$).

We now consider the only case (besides $\uparrow^{\exists x A(x)}$ below) where using shift and reset, or our Kripke-style models, is crucial. Let $\Gamma\vdash A\vee B$ be a neutral derivation. Let a formula $C$ and $\Gamma'\ge\Gamma$ be given, and let
\[\tag{\#}
\icont{C}{\Gamma'}{\Gamma''\Vdash A \text{ or } \Gamma''\Vdash B}.
\]
We prove $\Gamma'\vdash C$ by the following derivation tree:
\begin{prooftree}
\axc{$\Gamma\vdash A\vee B$}
\uic{$\Gamma'\vdash A\vee B$}{}
\axc{$A\in A,\Gamma'$}
\uic{$A,\Gamma'\vdash A$}{\textsc{Ax}}
\uic{$A,\Gamma'\Vdash A$}{($\uparrow$)}
\uic{$A,\Gamma'\Vdash A$ or $A,\Gamma'\Vdash B$}{$\inl$}
\uic{$A,\Gamma'\vdash C$}{(\#)}
\axc{$B\in B,\Gamma'$}
\uic{$B,\Gamma'\vdash B$}{\textsc{Ax}}
\uic{$B,\Gamma'\Vdash B$}{($\uparrow$)}
\uic{$B,\Gamma'\Vdash A$ or $B,\Gamma'\Vdash B$}{$\inr$}
\uic{$B,\Gamma'\vdash C$}{(\#)}
\tic{$\Gamma'\vdash C$}{$\vee_E$}
\end{prooftree}

\emph{Induction case for $\Rightarrow$.} Let $\Gamma\Vdash A\Rightarrow B$ i.e.
\[
\idcont{C}{\Gamma}{\left(\forall \Gamma_3\ge\Gamma''.~ \Gamma_3\Vdash A \to \Gamma_3\Vdash B\right)}.
\]
We apply this hypothesis by setting $C:=A\Rightarrow B$ and $\Gamma':=\Gamma$, and then, given $\Gamma''\ge\Gamma$ s.t. 
\[\tag{\#}
\forall \Gamma_3\ge\Gamma''.~ \Gamma_3\Vdash A \to \Gamma_3\Vdash B
\]
we have to derive $\Gamma''\vdash A\Rightarrow B$. This follows by applying $(\Rightarrow_I)$, the IH for($\downarrow$), then (\#), and finally the IH for ($\uparrow$) with the \textsc{Ax} rule.

Let $\Gamma\vdash A\Rightarrow B$ be a neutral derivation. We prove $\Gamma\Vdash A\Rightarrow B$ by applying unit (Lemma \ref{intcomp_model_monad}), and then, given $\Gamma'\ge\Gamma$ and $\Gamma'\Vdash A$, we have to show that $\Gamma'\Vdash B$. This is done by applying the IH for ($\uparrow$) on the $(\Rightarrow_E)$ rule, with the IH for ($\downarrow$) applied to $\Gamma'\Vdash A$. 

\emph{Induction case for $\forall$.} We recall that the domain function $D(-)$ is constant in the universal model $\calU$. Let $\Gamma\Vdash \forall x A(x)$ i.e.
\[
\idcont{C}{\Gamma}{\left(\forall \Gamma_3\ge\Gamma''.~ \forall t\in D.~ \Gamma_3\Vdash A(t)\right)}.
\]
We apply this hypothesis by setting $C:=\forall x A(x)$ and $\Gamma':=\Gamma$, and then, given $\Gamma''\ge\Gamma$ s.t. 
\[\tag{\#}
\forall \Gamma_3\ge\Gamma''.~\forall t\in D.~ \Gamma_3\Vdash A(t)
\]
we have to derive $\Gamma''\vdash \forall x A(x)$. This follows by applying $(\forall_I)$, the IH for($\downarrow$), and then (\#).

Let $\Gamma\vdash \forall x A(x)$ be a neutral derivation. We prove $\Gamma\Vdash \forall x A(x)$ by applying unit (Lemma \ref{intcomp_model_monad}), and then, given $\Gamma'\ge\Gamma$ and $t\in D$, we have to show that $\Gamma'\Vdash A(t)$. This is done by applying the IH for ($\uparrow$) on the $(\forall_E)$ rule and the hypothesis $\Gamma\vdash \forall x A(x)$.

\emph{Induction case for $\exists$.} Let $\Gamma\Vdash \exists x A(x)$ i.e.
\[
\idcont{C}{\Gamma}{\left(\exists t\in D.~ \Gamma''\Vdash A(t)\right)}.
\]
We apply this hypothesis by setting $C:=\exists x A(x)$ and $\Gamma':=\Gamma$, and then, given $\Gamma''\ge\Gamma$ s.t. $\exists t\in D.~ \Gamma''\Vdash A(t)$, we have to derive $\Gamma''\vdash \exists x A(x)$. This follows by applying $(\exists_I)$ with $t\in D$, and the IH for($\downarrow$).

Let $\Gamma\vdash \exists x A(x)$ be a neutral derivation. Let a formula $C$ and $\Gamma'\ge\Gamma$ be given, and let
\[\tag{\#}
\icont{C}{\Gamma'}{\exists t\in D. \Gamma''\Vdash A(t)}.
\]
We prove $\Gamma'\vdash C$ by the following derivation tree:
\begin{prooftree}
\axc{$\Gamma\vdash \exists x A(x)$}
\uic{$\Gamma'\vdash \exists x A(x)$}{}
\axc{$A(x)\in A(x),\Gamma'$}
\uic{$A(x),\Gamma'\vdash A(x)$}{\textsc{Ax}}
\uic{$A(x),\Gamma'\Vdash A(x)$}{($\uparrow$)}
\uic{$A(x),\Gamma'\vdash C$}{(\#)}
\axc{$x$-fresh}
\tic{$\Gamma'\vdash C$}{$\exists_E$}
\end{prooftree}

\emph{The result of reification ``$\downarrow$'' is in normal form.} By inspection of the proof.\qed
\end{proof}

\section{Normalisation by Evaluation in IK-CPS Models}\label{sec:correctness}

In this section we give the algorithm that we manually extracted from the Coq formalisation, for the restriction to the interesting propositional fragment that involves implication and disjunction. The algorithm extracted automatically by Coq contains too many details to be instructive. 

The following evaluation function for $\lambda^{\to\vee}$-terms is behind the proof of Theorem~\ref{int_soundness}:
\begin{align*}
\llbracket \Gamma\vdash p:A \rrbracket_{w\Vdash\Gamma} &:  w\Vdash A\\
\\
\llbracket a \rrbracket_\rho &:= \rho(a) \\
\llbracket\lambda a.p\rrbracket_\rho &:= \kappa\mapsto\kappa\cdot\left(\alpha\mapsto\llbracket p\rrbracket_{\rho,a\mapsto\alpha}\right) = \eta\cdot\left(\alpha\mapsto\llbracket p\rrbracket_{\rho,a\mapsto\alpha}\right)\\
\llbracket p q\rrbracket_\rho &:= \kappa\mapsto\llbracket p\rrbracket_\rho\cdot\left(\phi\mapsto\phi\cdot\llbracket q\rrbracket_\rho\cdot\kappa\right)\\
\llbracket \iota_1 p\rrbracket_\rho &:= \kappa\mapsto\kappa\cdot\left(\inl\cdot\llbracket p\rrbracket_\rho\right) = \eta\cdot\left(\inl\cdot\llbracket p\rrbracket_\rho\right)\\
\llbracket \iota_2 p\rrbracket_\rho &:= \kappa\mapsto\kappa\cdot\left(\inr\cdot\llbracket p\rrbracket_\rho\right) = \eta\cdot\left(\inr\cdot\llbracket p\rrbracket_\rho\right)\\
\llbracket\caseof{p}{a_1}{q_1}{a_2}{q_2}\rrbracket_\rho &:= \kappa\mapsto\llbracket p\rrbracket_\rho\cdot\left(\gamma\mapsto
\left\{
  \begin{array}{ll}
    \llbracket q_1\rrbracket_{\rho,a_1\mapsto\alpha}\cdot\kappa & \text{ if } \gamma=\inl\cdot\alpha\\
    \llbracket q_2\rrbracket_{\rho,a_2\mapsto\beta}\cdot\kappa & \text{ if } \gamma=\inr\cdot\beta
  \end{array}
\right.\right)
\end{align*}

The following is the algorithm behind Theorem~\ref{int_completeness}:
\begin{align*}
\downarrow_\Gamma^A &: \Gamma\Vdash A \to \left\{p\in\Lambda\text{-nf } ~|~ \Gamma\vdash p:A\right\}\\
\uparrow_\Gamma^A &: \left\{e\in\Lambda\text{-ne } ~|~ \Gamma\vdash e:A\right\} \to \Gamma\Vdash A\\
\\
\downarrow_\Gamma^X &:= \alpha\mapsto\mu\cdot\alpha & X\text{-atomic}\\
\uparrow_\Gamma^X &:= e\mapsto \eta\cdot e & X\text{-atomic}\\
\downarrow_\Gamma^{A\Rightarrow B} &:= \eta\cdot\left(\phi\mapsto \lambda a.\downarrow_{\Gamma,a:A}^B\left(\phi\cdot\uparrow_{\Gamma,a:A}^A a\right)\right) & a\text{-fresh}\\
\uparrow_\Gamma^{A\Rightarrow B} &:= e\mapsto \eta\cdot\left(\alpha\mapsto\uparrow_\Gamma^B\left(e\left(\downarrow_\Gamma^A \alpha\right)\right)\right)\\
\downarrow_\Gamma^{A\vee B} &:= \eta\cdot\left(\gamma\mapsto
\left\{
  \begin{array}{ll}
    \iota_1 \downarrow_\Gamma^A\alpha &\text{ if } \gamma=\inl\cdot\alpha\\
    \iota_2 \downarrow_\Gamma^B\beta &\text{ if } \gamma=\inr\cdot\beta
  \end{array}
\right.
\right)\\
\uparrow_\Gamma^{A\vee B} &:= e\mapsto\kappa\mapsto\caseof{e}{a_1}{\kappa\cdot\left(\inl\cdot\uparrow_{\Gamma,a_1:A}^A a_1\right)}{a_2}{\kappa\cdot\left(\inr\cdot\uparrow_{\Gamma,a_2:B}^B a_2\right)} & a_i\text{-fresh}
\end{align*}

\section{Variants and Relation to Kripke Models}\label{sec:variants}

\subsection{``Call-by-value'' Models} Defining forcing on composite formulae in Definition~\ref{intcomp_model_def} proceeds analogously to defining the call-by-name CPS translation \cite{Plotkin1975}, or Kolmogorov's double-negation translation \cite{TroelstraVD1,MurthyThesis}. A definition analogous to the ``call-by-value''  CPS translation \cite{Plotkin1975} is also possible, by defining (non-strong) forcing by:

\begin{itemize}
\item $w\svd A\wedge B$ if $w\svd A$ and $w\svd B$;
\item $w\svd A\vee B$ if $w\svd A$ or $w\svd B$;
\item $w\svd A\Rightarrow B$ if for all $w'\ge w$, $w\svd A$ implies $w\Vdash B$;
\item $w\svd \forall x. A(x)$ if for all $w'\ge w$ and all $t\in D(w')$, $w'\Vdash A(t)$;
\item $w\svd \exists x. A(x)$ if $w\svd A(t)$ for some $t\in D(w)$.
\end{itemize}

One can prove this variant of IK-CPS models sound and complete, similarly to Section~\ref{sec:completeness}, except for two differences. Firstly, in the statement of Soundness, one needs to put $w\svd \Gamma$ in place of $w\Vdash\Gamma$. Secondly, due to the first difference, the composition of soundness of completeness that gives normalisation works for \emph{closed} terms only.

\subsection{Classical Models}

In \cite{intcomp_formalisation,IlikThesis,IlikLH2010}, we presented the following notion of model which is complete for \emph{classical} predicate logic and represents an NBE algorithm for it.

\begin{definition}\label{cbv_clkr} A \emph{Classical Kripke-CPS model (CK-CPS)}, is given by:
  \begin{itemize}
  \item a preorder $(K, \le)$ of \emph{possible worlds};
  \item a \fbox{unary} relation on worlds $(-) \exploding$ labelling a world as \emph{exploding};
  \item a binary relation $(-) \svd (-)$ of \emph{strong forcing} between worlds and atomic formulae, such that
    \[
    \text{ for all } w'\ge w, w\svd X \to w'\svd X,
    \]
  \item and a domain of quantification $D(w)$ for each world $w$, such that
    \[
    \text{ for all } w'\ge w, D(w)\subseteq D(w').
    \]
  \end{itemize}
  The relation $(-) \svd (-)$ of \emph{strong forcing} is \emph{extended from atomic to composite formulae} inductively and by simultaneously defining \fbox{two} new relations, refutation and (non-strong) forcing:
  \begin{itemize}
  \item[$\star$] A formula $A$ is \emph{refuted} in the world $w$ (notation $w:A\Vdash$) if any world $w'\ge w$, which strongly forces $A$, is exploding;
  \item[$\star$] A formula $A$ is \emph{forced} in the world $w$ (notation $w\Vdash A$) if any world $w'\ge w$, which refutes $A$, is exploding;
  \item $w\svd A\wedge B$ if $w\Vdash A$ and $w\Vdash B$;
  \item $w\svd A\vee B$ if $w\Vdash A$ or $w\Vdash B$;
  \item $w\svd A\Rightarrow B$ if for all $w'\ge w$, $w\Vdash A$ implies $w\Vdash B$;
  \item $w\svd \forall x. A(x)$ if for all $w'\ge w$ and all $t\in D(w')$, $w'\Vdash A(t)$;
  \item $w\svd \exists x. A(x)$ if $w\Vdash A(t)$ for some $t\in D(w)$.
  \end{itemize}
\end{definition}

The differences between Definition~\ref{intcomp_model_def} and Definition~\ref{cbv_clkr} are marked with boxes. We can also present CK-CPS using binary exploding nodes, by defining $w\svd\bot := \forall C. w\explod{C}$. Then, we get the following statement of forcing in CK-CPS,
\[
\forall w'\ge w.~ \left(\forall w''\ge w'.~ w''\sforces A \to \forall I. w''\explod{I}\right) \to \forall O. w'\explod{O},
\]
versus forcing in IK-CPS,
\[
\forall C.~ \forall w'\ge w.~ \left(\forall w''\ge w'.~ w''\sforces A \to w''\explod{C}\right) \to w'\explod{C}.
\]

The difference between forcing in the intuitionistic and classical models is, then, that: 1) the dependency on $C$ is necessary in the intuitionistic case, while it is optional in the classical case; 2) the continuation (the internal implication) in classical forcing is allowed to change the parameter $C$ upon application, whereas in intuitionistic forcing the parameter is not local to the continuation, but to the continuation of the continuation.

At this point we also remark that the use of dependent types to handle the parameter $C$ is determined by the fact that we formalise our definitions in Intuitionistic Type Theory. Otherwise, the quantification $\forall C. \cdots$ is quantification over first-order individuals, for example natural numbers.

\subsection{Kripke Models}

Let $A(n)$ be an arbitrary first-order formula and let $X(n,m)$ be a $\Sigma^0_1$-formula. Denote the following arithmetic schema by (D-DNS$^+$) for ``dependent Double-negation Shift schema, strengthened''.
\begin{prooftree}\label{ddnsplus}
  \axc{$\forall m.~ \forall n_1\ge n.~ \left(\forall n_2\ge n_1.~ A(n_2) \to X(n_2,m)\right) \to X(n_1,m)$}
  \uic{$A(n)$}{D-DNS$^+$}
\end{prooftree}

\begin{proposition} Let $\mathcal{K}=(K,\le,D,\vDash,\vDash_\bot)$ be any structure such that $\vDash$ denotes forcing in the standard Kripke model arising from $\mathcal{K}$, and $\Vdash$ denotes (non-strong) forcing in the IK-CPS  model arising from the same $\mathcal{K}$.

Then, in the presence of (D-DNS$^+$) at meta-level, for all formula $A$, and any $w\in K$,
\[
w\vDash A \longleftrightarrow w\Vdash A.
\]
\end{proposition}
\begin{proof}
The proof is by induction on $A$, using (D-DNS$^+$) to prove,
\begin{prooftree}
  \axc{$\forall C.~ \forall w_1\ge w.~ \left(\forall w_2\ge w_1.~ \left(w_2\Vdash A \text{ or } w_2\Vdash B\right) \to w_2\explod{C})\right) \to w_1\explod{C}$}
  \uic{$w\Vdash A \text{ or } w\Vdash B$}{,}
\end{prooftree}
needed in the case for disjunction, and similarly for the existential quantifier.
\end{proof}

\begin{corollary} Completeness of full intuitionistic predicate logic with respect to standard Kripke models is provable constructively, in the presence of D-DNS$^+$.
\end{corollary}

\begin{remark} It is the other direction of this implication that Kreisel proved, for a specialisation of D-DNS$^+$. (Section~\ref{sec:nbe}) To investigate more precisely whether D-DNS$^+$ captures exactly constructive provability of completeness for Kripke models remains future work.
\end{remark}

\section{Conclusion}\label{sec:conclusion}

We emphasised that our algorithm is $\beta$-NBE, because were we able to identify $\beta\eta$-equal terms of $\lambda^{\to\vee}$ through our NBE function, we would have solved the problem of the existence of canonical $\eta$-long normal form for $\lambda^{\to\vee}$. However, as shown by \cite{FioreCB2006}, due to the connection with Tarski's High School Algebra Problem \cite{BurrisL1993,Wilkie}, the notion of such a normal form is not finitely axiomatisable. If one looks at examples of $\lambda^{\to\vee}$-terms which are $\beta\eta$-equal but are not normalised to the same term by Danvy's (and our) algorithm, one can see that in the Coq type theory these terms are interpreted as denotations that involve commutative cuts.

In recent unpublished work \cite{Danvy2008}, Danvy also developed a version of his NBE algorithm directly in CPS, without using delimited control operators.

In \cite{Barral2009}, Barral gives a program for NBE of $\lambda$-calculus with sums by just using the exceptions mechanism of a programming language, which is something \textit{a priori} strictly weaker than using delimited control operators.

In \cite{AltenkirchDHS}, Altenkirch, Dybjer, Hofmann, and Scott, give a topos theoretic proof of NBE for a typed $\lambda$-calculus with sums, by constructing a sheaf model. The connection between sheaves and Beth semantics\footnote{We remark that, for the fragment $\{\Rightarrow,\forall,\wedge\}$, NBE can also be seen as completeness for \emph{Beth} semantics, since forcing in Beth and Kripke models is the same thing on that fragment.}
 is well known. While the proof is constructive, due to their use of topos theory, we were unable to extract an algorithm from it.

In \cite{MacedonioS}, Macedonio and Sambin present a notion of model for extensions of Basic logic (a sub-structural logic more primitive than Linear logic), which, for intuitionistic logic, appears to be related to our notion of model. However, they demand that their set of worlds $K$ be saturated, while we do not, and we can hence also work with finite models. 

In \cite{Filinski2001}, Filinski proves the correctness of an NBE algorithm for Moggi's computational $\lambda$-calculus, including sums. We found out about Filinski's paper right before finishing our own. He also evaluates the input terms in a domain based on continuations.

\section*{Acknowledgements} To Hugo Herbelin for inspiring discussions and, in particular, for suggesting to try polymorphism, viz. Remark \ref{polymorphism}. To Olivier Danvy for suggesting reference \cite{Filinski2001}.

\bibliographystyle{plain}
\bibliography{intcomp}

\end{document}

%% file: macros.tex
\theoremstyle{definition}
\newtheorem{definition}{Definition}[section]
\newtheorem*{definition*}{Definition}

\theoremstyle{plain}
\newtheorem{proposition}[definition]{Proposition}
\newtheorem{lemma}[definition]{Lemma}
\newtheorem{theorem}[definition]{Theorem}
\newtheorem{corollary}[definition]{Corollary}
\newtheorem*{lemma*}{Lemma}
\newtheorem*{theorem*}{Theorem}
\newtheorem*{corollary*}{Corollary}

\theoremstyle{remark}
\newtheorem{remark}[definition]{Remark}

\newcommand{\axc}[1]{\AxiomC{#1}}
\newcommand{\uic}[2]{\RightLabel{\small{#2}}\UnaryInfC{#1}}
\newcommand{\bic}[2]{\RightLabel{\small{#2}}\BinaryInfC{#1}}
\newcommand{\tic}[2]{\RightLabel{\small{#2}}\TrinaryInfC{#1}}

\newcommand{\caseof}[5]{\mathsf{case} ~ #1 ~\mathsf{of} ~ \left(#2.#3 \| #4.#5\right)}
\newcommand{\destas}[4]{\mathsf{dest} ~ #1 ~\mathsf{as} ~ \left(#2.#3\right) ~\mathsf{in}~ #4}
\newcommand{\shift}[2]{\mathcal{S}#1.#2}
\newcommand{\reset}[1]{\# #1}

\newcommand{\calU}{\mathcal{U}}

\newcommand{\sforces}{\Vdash_{\!\!\! s}}
\newcommand{\svd}{\Vdash_{\!\!\! s}}
\newcommand{\exploding}{\Vdash_{\!\!\bot}}
\newcommand{\explod}[1]{\Vdash^{\!\!#1}_{\!\!\bot}}

\DeclareMathOperator{\inl}{inl}
\DeclareMathOperator{\inr}{inr}

\newcommand{\icont}[3]{\forall #2'\ge #2.~ \left(#3 \to #2'\vdash #1\right)}
\newcommand{\idcont}[3]{\forall C.~ \forall #2'\ge #2.~ \left(\left(\forall #2''\ge #2'.~ #3 \to #2''\vdash #1\right) \to #2'\vdash #1\right)}
